\documentclass[12pt,reqno]{amsart}
\usepackage{amsmath,amsthm,amsfonts}

\newtheorem{thm}{Theorem}
\newtheorem{lem}{Lemma}

\DeclareMathOperator{\tr}{Tr}

\newcommand{\R}{\mathbb{R}}

\newcommand{\Z}{\mathbb{Z}}
\newcommand{\RR}{\mathcal{R}}
\newcommand{\cS}{\mathcal{S}}

\newcommand{\A}{\mathcal{A}}
\newcommand{\T}{\mathcal{T}}
\newcommand{\PP}{\mathcal{P}}

\begin{document}
\title[Butson Hadamard Matrices]{Butson-type complex Hadamard matrices \\
and association schemes on Galois rings of characteristic 4}
\author{Takuya Ikuta}
\address{Kobe Gakuin University, Kobe, 650-8586, Japan}
\email{ikuta@law.kobegakuin.ac.jp}
\author{Akihiro Munemasa}
\address{Tohoku University, Sendai, 980-8579, Japan}
\email{munemasa@math.is.tohoku.ac.jp}
\thanks{The work of A.M. was supported by JSPS KAKENHI grant number 17K05155.}

\date{December 17, 2017}

\begin{abstract}
We consider nonsymmetric hermitian complex Ha\-da\-mard matrices belonging to
the Bose--Mesner algebra of commutative nonsymmetric association schemes.
First, we give a characterization of  the eigenmatrix of 
a commutative  nonsymmetric association scheme of class $3$
whose Bose--Mesner algebra contains  a nonsymmetric hermitian complex Hadamard matrix,
and show that such a complex Hadamard matrix is necessarily a Butson-type complex Hadamard matrix 
whose entries are $4$-th roots of unity.
We also give nonsymmetric association schemes $\mathfrak{X}$ of
class $6$ on Galois rings of characteristic 4, and
classify hermitian complex Hadamard matrices belonging to
the Bose--Mesner algebra of $\mathfrak{X}$.
It is  shown that such a matrix is again
necessarily a Butson-type complex Hadamard matrix
whose entries are $4$-th roots of unity.
\end{abstract}
\maketitle

\section{Introduction} \label{sec.Intro}
A complex Hadamard matrix is a square matrix 
$W$ of order $n$ which satisfies $W\overline{W}^\top= nI$ and
all of whose entries are complex numbers of absolute value $1$.
A complex Hadamard matrix is said to be Butson-type, 
if all of its entries are roots of unity.
In our earlier work \cite{MI}, we proposed a method to classify symmetric
complex Hadamard matrices belonging to the Bose--Mesner algebra
of a symmetric association scheme.
In this paper, we propose an analogous method to classify nonsymmetric
hermitian complex Hadamard matrices belonging to the Bose--Mesner algebra
of a commutative nonsymmetric association scheme.
First we consider nonsymmetric hermitian complex Hadamard matrices
belonging to the Bose--Mesner algebra of 
a commutative nonsymmetric association scheme of class $3$,
and then we show that the first eigenmatrix of such a scheme is characterized as
\begin{equation}\label{d3P1}
\begin{pmatrix}
1&a(2a-1)&a(2a-1)&2a-1\\
1&ai&-ai&-1\\
1&-ai&ai&-1\\
1&-a&-a&2a-1
\end{pmatrix},
\end{equation}
where $a$ is a positive integer and $i=\sqrt{-1}$.
Moreover, we show that such a complex Hadamard matrix is necessarily
a Butson-type complex Hadamard matrix 
whose entries are $4$-th roots of unity.
An association scheme with the first eigenmatrix (\ref{d3P1}) is
a nonsymmetric amorphous association scheme
belonging to $L_{2a;1}$, according to \cite{IMY}.
An example of such an association scheme was 
constructed from a Galois ring of characteristic $4$  in \cite[Theorem 9]{IMY}.
Galois rings have been used to construct certain association
schemes. Yamada \cite{Y} used Galois rings of characteristic $4$
to construct distance-regular digraphs, which are nonsymmetric 
association schemes of class $3$. This construction was generalized in
\cite{IMY} to produce amorphous association schemes. Ma \cite{Ma} 
gave association schemes of class $3$ on Galois rings of characteristic $4$,
which do not come from fusions in any amorphous association schemes.
Moreover, certain properties of association schemes
obtained from Galois rings of odd characteristic
have been investigated in \cite{EP}.

Our second result of this paper is a construction of
nonsymmetric association schemes $\mathfrak{X}$ of
class $6$ on Galois rings of characteristic 4, 
whose first eigenmatrix is given by
\begin{equation} \label{P6}
(p_{i,j})_{\substack{0\leq i\leq 6 \\ 0\leq j\leq 6}}=\begin{pmatrix}
1&2b(b-1)&2b(b-1)&b&b&b-1&b\\
1&bi&-bi&0&0&-1&0\\
1&-bi&bi&0&0&-1&0\\
1&0&0&bi&-bi&b-1&-b \\
1&0&0&-bi&bi&b-1&-b \\
1&-2b&-2b&b&b&b-1&b \\
1&0&0&-b&-b&b-1&b
\end{pmatrix},
\end{equation}
where $b$ is a power of $4$.
We also classify hermitian complex Hadamard matrices belonging to
the Bose--Mesner algebra of $\mathfrak{X}$.
We show that such a matrix is necessarily a Butson-type matrix
whose entries are $4$-th roots of unity.

\section{Association schemes and complex Hadamard matrices} \label{sec.AS}
In this section,
we consider hermitian matrices
belonging to the Bose--Mesner algebra of a commutative association scheme.
Assuming that all
entries are complex numbers with absolute value $1$, 
we find conditions under which such a matrix is a complex Hadamard matrix.
We refer the reader to \cite{BI,BCN} for undefined terminology in the
theory of association schemes.

Let $X$ be a finite set with $n$ elements, and let
$\mathfrak{X}=(X,\{R_i\}_{i=0}^d)$ be a commutative association scheme 
with the first eigenmatrix 
$P=(P_{i,j})$.
We let $\mathfrak{A}$ denote the Bose--Mesner algebra
spanned by the adjacency matrices $A_0,A_1,\ldots,A_d$ of $\mathfrak{X}$.
Then the adjacency matrices are expressed as
\begin{equation}\label{eq:pij}
A_j=\sum_{i=0}^d P_{i,j}E_i \quad (j=0,1,\ldots,d),
\end{equation}
where $E_0=\frac{1}{n}J,E_1,\ldots,E_d$ are the primitive idempotents of $\mathfrak{A}$.
Let $w_0=1$ and $w_1,\ldots,w_d$ be complex numbers with absolute value $1$.
Set
\begin{equation}\label{eq:W}
W=\sum_{j=0}^d w_jA_j\in\mathfrak{A}.
\end{equation}
Define
\begin{equation}\label{eq:gamma}
\gamma_k=\sum_{j=0}^d w_jP_{k,j} \quad (k=0,1,\ldots,d).
\end{equation}
By \eqref{eq:pij}, \eqref{eq:W} and \eqref{eq:gamma} we have
\begin{equation}\label{eq:WE}
W=\sum_{k=0}^d \gamma_kE_k.
\end{equation}
Let $X_j$ {\rm($1\leq j\leq d$)} be indeterminates. For $k=1,2,\dots,d$, 
let $e_k$ be the polynomial defined by
\begin{equation}\label{eq:ek}
e_k=1+2\left(\sum_{j=1}^d P_{k,j}X_j+
\sum_{1\leq j_1<j_2\leq d}P_{k,j_1}P_{k,j_2}X_{j_1}X_{j_2}\right)
+\sum_{j=1}^dP_{k,j}^2X_j^2-n.
\end{equation}
Then we have the following.

\begin{lem}\label{lem:equiv}
Assume that the matrix $W$ given in \eqref{eq:W} is a hermitian matrix.
Then the following statements are equivalent:
\begin{itemize}
\item[\rm{(i)}] $W$ is a complex Hadamard matrix,
\item[\rm{(ii)}] $\gamma_k^2=n$ for $k=1,\ldots,d$,
\item[\rm{(iii)}] $(w_j)_{1\leq j\leq d}$ is a common zero of $e_k$ {\rm($k=1,\ldots,d$)}.
\end{itemize}
\end{lem}
\begin{proof}
By \eqref{eq:WE} we have $W\overline{W}^T=W^2=\sum_{k=0}^d\gamma_k^2E_k$.
Therefore, (i) implies (ii).
To prove the converse, it is enough to show that $\gamma_0^2=n$.
Since
\begin{align*}
(W^2)_{j,j}&=\sum_{k=1}^nW_{j,k}\overline{W_{j,k}}\\
&=n,
\end{align*}
the diagonal entries of $W^2$ are all $n$.
Thus,
\begin{align*}
n^2&=\tr W^2\\
&=\sum_{k=0}^d\gamma_k^2\tr E_k\\
&=\gamma_0^2+\sum_{k=1}^dn\tr E_k\\
&=\gamma_0^2+n\tr(I-E_0)\\
&=\gamma_0^2+n(n-1).
\end{align*}
Hence $\gamma_0^2=n$.
By \eqref{eq:gamma} we have
\[
\gamma_k^2=1+2\left(\sum_{j=1}^d P_{k,j}w_j+
\sum_{1\leq j_1<j_2\leq d}P_{k,j_1}P_{k,j_2}w_{j_1}w_{j_2}\right)
+\sum_{j=1}^dP_{k,j}^2w_j^2
\]
for $k=1,\dots,d$.
Therefore, the equivalence of (ii) and (iii) follows.
\end{proof}

\section{Hermitian complex Hadamard matrices 
 and nonsymmetric association schemes of class $3$}\label{sec.3-class}
In this section, we classify nonsymmetric hermitian complex Hadamard matrices 
belonging to the Bose--Mesner algebra of a nonsymmetric commutative association scheme 
$\mathfrak{X}=(X,\{R_i\}_{i=0}^3)$ of class $3$,
where $R_1^\top=R_2$, $R_3$ symmetric.
We use the following result.

\begin{lem}[{\cite[Lemma in (5.3)]{S}}]\label{lem:song}
Let 
\[
\begin{pmatrix}
1&\frac{k_1}{2}&\frac{k_1}{2}&k_2\\
1&\frac12(r+bi)&\frac12(r-bi)&-(r+1)\\
1&\frac12(r-bi)&\frac12(r+bi)&-(r+1)\\
1&\frac{s}{2}&\frac{s}{2}&-(s+1)
\end{pmatrix}
\]
be the first eigenmatrix of $\mathfrak{X}$, where $r,s$ are integers and $i^2=-1$.
Then one of the following holds.
\begin{itemize}
\item[\rm{(i)}] $(r,s,b^2)=(0,-(k_2+1),\frac{k_1(k_2+1)}{k_2})$,
\item[\rm{(ii)}] $(r,s,b^2)=(-(k_2+1),0,(1+k_2)(1+k_1+k_2))$,
\item[\rm{(iii)}] $(r,s,b^2)=(-1,k_1,k_1+1)$.
\end{itemize}
\end{lem}

Note that, if we put
\begin{equation}\label{eq:0930}
(k_1,k_2)=(2a(2a-1),2a-1),
\end{equation}
in part (i) of Lemma~\ref{lem:song}, 
we have the matrix (\ref{d3P1}).

Let $\mathfrak{A}$ be the Bose--Mesner algebra of $\mathfrak{X}$ 
which is the linear span of the adjacency matrices
$A_0,A_1,A_2,A_3$ of $\mathfrak{X}$, where $A_1^\top=A_2$, $A_3$ symmetric.
Let $w_0=1$ and $w_1,w_2,w_3$ be complex numbers of absolute value $1$.
Set
\begin{equation} \label{eq:w0W}
W=w_0A_0+w_1A_1+w_2A_2+w_3A_3\in\mathfrak{A}.
\end{equation}
Then our first main theorem is the following.

\begin{thm}\label{thm:main0925}
Assume that the matrix \eqref{eq:w0W} is a 
hermitian complex Ha\-da\-mard matrix and not a real Ha\-da\-mard matrix.
Then $\mathfrak{X}$ is a nonsymmetric association scheme 
whose unique nontrivial symmetric relation 
consists of $2a$ cliques of size $2a$,
and the first eigenmatrix of $\mathfrak{X}$ is given by \eqref{d3P1}. 
Moreover, $w_1=\pm i$ and $w_3=1$.
\end{thm}

\begin{proof}
Suppose that the matrix $W$ is a complex Hadamard matrix.
Without loss of generality, we may assume $w_0=1$ in (\ref{eq:w0W}).
Since $W$ is a hermitian matrix, we have $w_1w_2=1$ and $w_3=\pm1$.
Since $W$ is not a real Hadamard matrix, we have
$w_1-w_2\not=0$.
By Lemma~\ref{lem:equiv}, $(w_i)_{i=1}^3$ is a common zero of the polynomials $e_k$ ($k=1,2,3$) defined in (\ref{eq:ek}).
Since
\[ e_1-e_2=bi(X_1-X_2)(r(X_1+X_2-2X_3)-2X_3+2), \]
we have
\begin{equation}\label{eq:1-1}
r(w_1+w_2-2w_3)-2w_3+2=0.
\end{equation}

First assume that $w_3=1$.
Then by (\ref{eq:1-1}) we have $r(w_1+w_2-2)=0$.
Hence $r=0$.
Therefore we have case (i) of Lemma~\ref{lem:song}.
Then 
\begin{equation}\label{eq:nk1}
(k_1,k_2)=\left(\frac{(s+1)b^2}{s},-(s+1)\right).
\end{equation}
After specializing $X_3=1$, we have
\[
e_3-e_1=\frac{1}{4}((s-bi)X_1+(s+bi)X_2-2s)((s+bi)X_1+(s-bi)X_2-2s).
\]
Hence $((s-bi)w_1+(s+bi)w_2-2s)((s+bi)w_1+(s-bi)w_2-2s)=0$.
Then by $w_2=w_1^{-1}$ we have
\begin{equation}\label{eq:1-2}
(w_1-1)^2((s-bi)w_1-(s+bi))((s+bi)w_1-(s-bi))=0.
\end{equation}
Put $w=(s+2bi)/(s-2bi)$.
By $w_1\not=1$ we have $w_1\in\{w,\overline{w}\}$ by (\ref{eq:1-2}).
We may assume $w_1=w$.
After specializing $X_1=w$, $X_2=\overline{w}$ and $X_3=1$, we have
\[
e_1=-(1+k_1+k_2)-\frac{4b^4s^2}{(b^2+s^2)^2}.
\]
Hence
\begin{equation}\label{eq:1-3}
k_1+k_2=-1-\frac{4b^4s^2}{(b^2+s^2)^2}.
\end{equation}
Then by (\ref{eq:nk1}), (\ref{eq:1-3}) we have
\[
(b-s)(b+s)((s+1)b^4-(s-2)s^2b^2+s^4)=0.
\]
Since $b>0$ and $s=-(k_2+1)<-1$, we have $b-s\not=0$.
First assume that $s=-b$.
Then, by putting $b=2a$, we have $(k_1,k_2)=(2a(2a-1),2a-1)$ and $w_1=-i$.
Therefore we have the first eigenmatrix (\ref{d3P1}).
Next assume that
\begin{equation}\label{eq:0921}
(s+1)b^4-(s-2)s^2b^2+s^4=0.
\end{equation}
Then the discriminant of (\ref{eq:0921}) as a quadratic equation in $b^2$ is $s^4((s-4)^2-16)$.
Since $b^2$ is rational, $(s-4)^2-16$ is a square. This implies $s\in\{0,-1\}$.
This contradicts $s=-(k_2+1)<-1$.

Next assume that $w_3=-1$.
Then by (\ref{eq:1-1}) we have $r(w_1+w_2+2)+4=0$.
Put $g(x)=rx^2+2(r+2)x+r$.
By $w_2=w_1^{-1}$ we have $g(w_1)=0$.
Since $w_1\not\in\R$, we have $r+1<0$.
Thus we have case (ii) of Lemma~\ref{lem:song}.
After specializing $X_1=w_1$, $X_2=w_1^{-1}$, $X_3=-1$ and $s=0$, we have
\[
 e_1-e_3=\frac{(w_1+1)f(w_1)}{4w_1^2},
\]
where $f(x)=((r+bi)x+r-bi)((r+bi)x^2+2(r+4)x+r-bi)$.
Since $w_1\not=-1$, we have $f(w_1)=0$.
Since
\begin{align*}
f(x)&=\frac{(r+bi)((r+bi)rx+r(r+4)-(3r+4)bi)}{r^2}g(x)\\
&\quad-\frac{4((r+1)b^2+r^2)}{r^2}((r+4)x-r)
\end{align*}
and $(r+4)w_1-r\not=0$,
we have
\begin{equation}\label{eq:0918}
(r+1)b^2+r^2=0.
\end{equation}
Substituting $(r,b^2)=(-k_2+1,(1+k_1+k_2)(k_2+1))$ in (\ref{eq:0918}),
we have $(k_2+1)(k_2^2+k_1k_2-1)=0$.
This is a contradiction.
\end{proof}

In the next section,
we will construct a nonsymmetric association scheme $\mathfrak{X}$ of class $6$
on a Galois ring of characteristic $4$, with the first eigenmatrix
\eqref{P6}.
An association scheme with the first eigenmatrix (\ref{d3P1}) can be obtained
by fusing some relations of $\mathfrak{X}$.

\section{Association schemes on Galois rings} \label{sec.Galoisring}

For the reminder of this section, we let $e\geq 3$ be an odd positive integer.
We refer the reader to \cite{Y} for basic theory of Galois rings.
Let $F=\text{GF}(2)$ be the prime field of characteristic $2$ and
$K=\text{GF}(2^e)$ be an extension of degree $e$.
We identify $K$ with $F[x]/(\varphi(x))$, 
where $\varphi(x)$ is a primitive polynomial of degree $e$ over $F$,
and we denote by $\zeta$ a root of $\varphi(x)$ in $K$.

Let $\A=\Z/4\Z$.
There exists a monic polynomial $\Phi(x)$ over $\A$ such that $\Phi(x)\equiv \varphi(x)$ mod $2\A[x]$
and $\Phi(x)$ divides $x^{2^e-1}-1$ in $\A[x]$ (see \cite[Theorem 1]{CS}).
The ring $\RR=\A[x]/(\Phi(x))$ is called a Galois ring,
and it is a local ring with maximal ideal $\PP=2\RR$ and 
residue field $\RR/\PP\cong K$.

Let $\xi$ be the image of $x$ in $\RR$, so that $\xi+\PP$ is mapped to $\zeta$
under the isomorphism $\RR/\PP\cong K$.
Then $\RR=\A[\xi]$ and $\xi$ has order $2^e-1$ in $\RR$.
Let $\T$ be the cyclic group of order $2^e-1$ generated by $\xi$.
Since $\T$ is mapped bijectively onto $K\setminus\{0\}$ under the natural homomorphism $\RR\rightarrow K$,
each element $\alpha\in \RR$ is uniquely expressed as
\[
\alpha=\alpha_0+2\alpha_1, \quad \alpha_0,\alpha_1\in\T\cup\{0\}.
\]
The unit group $\RR^{\ast}$ is the direct product of $\T$ 
and the principal unit group $\mathcal{E}=1+\PP$.
Let $\T_\delta=\{\xi^j\mid 0\leq j\leq 2^e-2,\;\mathrm{Tr}(\zeta^j)=\delta\}$ ($\delta=0,1$).
We set $\PP_0=2\T_0\cup\{0\}$ and $H=1+\PP_0$.
Thus $\PP_0$ and $H$ are subgroups of index $2$ in the additive group $\PP$ and 
the multiplicative group $\mathcal{E}$, respectively.
Setting $b=2^{e-1}$, we have $b=|\PP_0|=|H|$.

The mapping $\xi\mapsto\xi^2$ can be extended to a ring automorphism $f$ of $\RR$
which fixes $\A$ elementwise.
The ring automorphism $f$ is called the {\em Frobenius automorphism}.
For $\alpha=\alpha_0+2\alpha_1$ ($\alpha_0,\alpha_1\in\T\cup\{0\}$), we have
\[
\alpha^f=\alpha_0^2+2\alpha_1^2,
\]
and the trace of $\alpha\in\RR$ is defined by
\[
\cS(\alpha)=\alpha+\alpha^f+\cdots+\alpha^{f^{e-1}}.
\]
Note that $\cS$ is an $\A$-linear mapping from $\RR$ to $\A$.
The additive characters of $\RR$ are given by $\alpha\mapsto\chi(\alpha\beta)$ 
($\beta\in\RR$), where
\[
\chi(\alpha)=i^{e\cS(\alpha)},
\]
and $i^2=-1$. Define
\[
\lambda_{\alpha}(A)=\sum_{\beta\in A}\chi(\alpha\beta),
\]
for $\alpha\in\RR$ and $A\subset\RR$.
Set $\lambda=\lambda_1$.

Observe
\begin{align}
\mathcal{E}&=H\cup(-H), \nonumber \\
\alpha\T&=\PP\setminus\{0\} \quad \text{for any $\alpha\in\PP\setminus\{0\}$}, \label{eq:aT} \\
\mathrm{Tr}(\zeta^j)&\equiv \cS(\xi^j)\bmod{2}. \nonumber
\end{align}
Since $e$ is odd, we obtain
\begin{align}
\chi(\alpha)&=
\begin{cases}
1 &  \ \text{if} \ \alpha\in\PP_0, \\
-1 &  \ \text{if} \ \alpha\in\PP\setminus\PP_0,
\end{cases}\label{chival} \\
\chi(\alpha)&=i \quad \text{for} \ \alpha\in H. \label{chiH}
\end{align}

\begin{lem}  \label{lem:aP0}
For any $\alpha\in\RR^{\ast}\setminus\mathcal{E}$ we have the following.
\begin{align}
\lambda(\alpha\mathcal{P}_0)&=0, \label{eq:aP0} \\
\lambda(\alpha H)&=0. \label{eq:aH}
\end{align}
\end{lem}
\begin{proof}
Note that the correspondence $2\xi^j\mapsto \zeta^j$ extends to an
isomorphism from $\PP$ to $K$ as additive groups. Under this isomorphism,
$\PP_0$ is mapped to $\mathrm{Tr}^{-1}(0)$. Since
$\alpha\notin\mathcal{E}$, the image $a=\alpha+\PP$ of $\alpha$ in $K=\RR/\PP$
is not $1$. By the non-degeneracy of the trace function on $K$, we have
$a\mathrm{Tr}^{-1}(0)\neq\mathrm{Tr}^{-1}(0)$. This implies 
$\alpha\mathcal{P}_0\neq\mathcal{P}_0$.
Then
\[A=\alpha\PP_0\cap\PP_0\]
is a subgroup of index $4$ in the additive group $\PP$.
By \eqref{chival} we have
\[
\chi(\beta)=
\begin{cases}
    1 & \text{if} \quad \beta\in A, \\
    -1 & \text{if} \quad \beta\in\alpha\PP_0\setminus A.
\end{cases}
\]
Therefore $\lambda(\alpha\PP_0)=0$.
Since $\alpha H=\alpha+\alpha\PP_0$, we have
$\lambda(\alpha H)=\lambda(\alpha)\lambda(\alpha\PP_0) 
=0$.
\end{proof}

Recall $b=|\PP_0|=|H|$.

\begin{lem} \label{lem:lamH}
We have the following.
\begin{itemize}
\item[\rm{(i)}] $\lambda(H)=bi$,
\item[\rm{(ii)}] $\lambda(\T H)=bi$.
\end{itemize}
\end{lem}
\begin{proof}
(i) Obvious from \eqref{chiH}.

(ii) Since $\T\setminus\{1\}\subset\RR^{\ast}\setminus\mathcal{E}$,
we have
$\lambda((\T\setminus\{1\})H)=0$
by \eqref{eq:aH}. Then the result follows from  (i).
\end{proof}

\begin{lem} \label{lem:lamaH}
We have
\[
\lambda_{\alpha}(H)=
\begin{cases}
b & \text{if} \quad \alpha\in\PP_0, \\
-b & \text{if} \quad \alpha\in\PP\setminus\PP_0.
\end{cases}
\]
\end{lem}
\begin{proof}
Let $\alpha\in\mathcal{P}$. Then $\alpha H=\{\alpha\}$.
Hence $\lambda_{\alpha}(H)=|H|\chi(\alpha)$, and 
the result follows from \eqref{chival}.
\end{proof}

\begin{lem}  \label{lem:lamaTH}
For any $\alpha\in\PP\setminus\{0\}$ we have
\[ \lambda_{\alpha}(\T H)=-b. \]
\end{lem}
\begin{proof}
By \eqref{eq:aT} we have
\begin{align*}
\lambda_{\alpha}(\T H)&=\sum_{\beta\in\PP\setminus\{0\}}\lambda_{\beta}(H) \\
&=\sum_{\beta\in\PP}\lambda_{\beta}(H)-\lambda_0(H) \\
&=-\lambda_0(H) && (\text{by Lemma~\ref{lem:lamaH}}) \\
&=-b.
\end{align*}
\end{proof}

Then our second main theorem is the following.
\begin{thm} \label{thm:1}
Let
\begin{align*}
S_0&=\{0\}, \\
S_1&=(\T\setminus\{1\})H, \\
S_2&=(\T\setminus\{1\})(-H), \\
S_3&=H, \\
S_4&=-H, \\
S_5&=\PP_0\setminus\{0\}, \\
S_6&=\PP\setminus\PP_0.
\end{align*}
Then $\lambda_{\alpha}(S_j)$ is constant for any $\alpha\in S_i$ {\rm($i=0,1,\ldots,6$)}.
Put $p_{i,j}=\lambda_{\alpha}(S_j)$ for $\alpha\in S_i$ {\rm($i,j=0,1,\ldots,6$)}.
Then $P=(p_{i,j})_{\substack{0\leq i\leq 6 \\ 0\leq j\leq 6}}$
 is given by {\rm(\ref{P6})}.
In particular, if we set
\[
R_j=\{(\alpha,\beta)\in\RR\times\RR \mid \alpha-\beta\in S_j \} \quad (j=0,1,\ldots,6),
\]
then $\mathfrak{X}=(\RR,\{R_i\}_{i=0}^6)$ is a nonsymmetric association scheme of class $6$
with the first eigenmatrix \eqref{P6}.
\end{thm}
\begin{proof}
It is easy to check that $|S_1|=|S_2|=2b(b-1)$, $|S_3|=|S_4|=b$, $|S_5|=b-1$, and $|S_6|=b$.
So we have $p_{0,j}$ ($j=0,1,\ldots,6$) as in \eqref{P6}.
It is trivial that $\lambda_{\alpha}(S_0)=1$ for any $\alpha\in S_0\cup S_1\cup \cdots\cup S_6$.
Hence $p_{j,0}=1$ for $j=0,1,\ldots,6$ as in \eqref{P6}.

We claim that
\begin{align}
\lambda_{\alpha}(S_1)&=\begin{cases}
bi & \text{if} \quad \alpha\in S_1,\\
0 & \text{if} \quad \alpha\in S_3,\\
-2b & \text{if} \quad \alpha\in S_5,\\
0 & \text{if} \quad \alpha\in S_6,
\end{cases} \label{case:1} 
\displaybreak[0]\\
\lambda_{\alpha}(S_3)&=\begin{cases}
0 & \text{if} \quad \alpha\in S_1,\\
bi & \text{if} \quad \alpha\in S_3,\\
b & \text{if} \quad \alpha\in S_5,\\
-b & \text{if} \quad \alpha\in S_6,
\end{cases} \label{case:2} 
\displaybreak[0]\\
\lambda_{\alpha}(S_5)&=\begin{cases}
-1 & \text{if} \quad \alpha\in S_1,\\
b-1 & \text{if} \quad \alpha\in S_3,\\
b-1 & \text{if} \quad \alpha\in S_5,\\
b-1 & \text{if} \quad \alpha\in S_6.
\end{cases} \label{case:3}
\end{align}

First we prove \eqref{case:1}.
Assume that $\alpha\in S_1$.
Then $\lambda_{\alpha}(\T H)=\lambda(\T H)=bi$ by Lemma~\ref{lem:lamH} (ii),
and $\lambda_{\alpha}(H)=\lambda(\alpha H)$.
Since $S_1\subset\RR^{\ast}\setminus\mathcal{E}$, we have $\lambda(\alpha H)=0$ by 
\eqref{eq:aH}.
Hence $\lambda_{\alpha}(S_1)=bi$.
Next assume that $\alpha\in S_3$.
Similarly, we have $\lambda_{\alpha}(\T H)=\lambda(\T H)$ and 
$\lambda_{\alpha}(H)=\lambda(H)$.
By Lemma~\ref{lem:lamH} (i), (ii) we have $\lambda_{\alpha}(S_1)=0$.
Next assume that $\alpha\in S_5\cup S_6$.
Then by Lemma~\ref{lem:lamaH} and Lemma~\ref{lem:lamaTH},
we have $\lambda_{\alpha}(S_1)=-2b$ or $0$ depending on 
$\alpha\in S_5$ or $\alpha\in S_6$.
This completes the proof of \eqref{case:1}.

Next we prove \eqref{case:2}.
Assume that $\alpha\in S_1$.
Then $\lambda_{\alpha}(S_3)=\lambda(\alpha S_3)$.
Since $S_1\subset\RR^{\ast}\setminus\mathcal{E}$, we have $\lambda(\alpha S_3)=0$ by 
\eqref{eq:aH}.
Next assume that $\alpha\in S_3$.
Then $\lambda_{\alpha}(S_3)=\lambda(S_3)=bi$ by Lemma~\ref{lem:lamH} (i).
Next assume that $\alpha\in S_5\cup S_6$.
Then by Lemma~\ref{lem:lamaH}, $\lambda_{\alpha}(S_3)=b$ or $-b$ depending on 
$\alpha\in S_5$ or $\alpha\in S_6$.
This completes the proof of \eqref{case:2}.

Finally we prove \eqref{case:3}.
Assume that $\alpha\in S_1$.
Since $S_1\subset\RR^{\ast}\setminus\mathcal{E}$, 
we have $\lambda_{\alpha}(S_5)=\lambda(\alpha(\PP_0\setminus\{0\}))=-1$ by \eqref{eq:aP0}.
Next assume that $\alpha\in S_3$.
Since $\alpha\mathcal{P}_0=\PP_0$, we have $\lambda_{\alpha}(S_5)
=\lambda(\PP_0\setminus\{0\})=|S_5|=b-1$.
Next assume that $\alpha\in S_5\cup S_6$.
Then $\alpha(\PP_0\setminus\{0\})=\{0\}$.
Hence $\lambda_{\alpha}(S_5)=|S_5|=b-1$.
This completes the proof of \eqref{case:3}.

From \eqref{case:1}--\eqref{case:3}
we have $p_{1,1}$, $p_{3,1}$, $p_{5,1}$, $p_{6,1}$, 
$p_{1,3}$, $p_{3,3}$, $p_{5,3}$, $p_{6,3}$, 
$p_{1,5}$, $p_{3,5}$, $p_{5,5}$, and $p_{6,5}$ as in \eqref{P6}.
Since $\{S_j\}_{j=0}^6$ is a partition of $\RR$,
we have $\sum_{j=0}^6\lambda_{\alpha}(S_j)=0$ for $\alpha\not=0$.
Thus, it is enough to check that 
$\lambda_{\alpha}(S_j)$ is a constant independent of $\alpha\in S_i$ for $i=1,\ldots,6$ and $j=1,\ldots,5$.
Since $S_2=-S_1$, we have $p_{2,j}=\overline{p_{1,j}}$ and $p_{j,2}=\overline{p_{j,1}}$ for $j=1,\ldots,6$.
Since $S_4=-S_3$, we have $p_{4,j}=\overline{p_{3,j}}$ and $p_{j,4}=\overline{p_{j,3}}$ for $j=1,\ldots,6$.
Therefore we find all $p_{i,j}$ as in \eqref{P6}
from \eqref{case:1}--\eqref{case:3}.
\end{proof}

Let $\mathfrak{X}=(X,\{R_i\}_{i=0}^d)$ be a commutative association scheme.
A partition $\Lambda_0, \Lambda_1, \ldots, \Lambda_e$ of the index set $\{0,1,\ldots,d\}$
of the association scheme $\mathfrak{X}$
is said to be {\it admissible} if $\Lambda_0=\{0\}$, $\Lambda_i\not=\emptyset$ 
$(1\leq i\leq e)$ and 
$\Lambda_i'=\Lambda_j$ for some $j$ ($1\leq j\leq e$),
where $\Lambda_i'=\{\alpha' \mid \alpha\in\Lambda_i\}$, $R_{\alpha'}=\{(x,y)\mid (y,x)\in R_{\alpha}\}$.
Let $R_{\Lambda_i}=\bigcup_{\alpha\in\Lambda_i}R_{\alpha}$.
If $\mathfrak{Y}=(X,\{R_{\Lambda_i}\}_{i=0}^e)$ becomes an association scheme, 
then it is called a {\it fusion} scheme of $\mathfrak{X}$.

\medskip\par\noindent
{\bf Bannai--Muzychuk criterion} (\cite{B, Muzychuk}).
Let $\mathfrak{X}$ be a commutative association scheme 
with the first eigenmatrix $P$.
Let $\{\Lambda_j\}_{j=0}^e$ be an admissible partition of the index set $\{0,1,\ldots,d\}$.
Then $\{\Lambda_j\}_{j=0}^e$ gives rise to a fusion scheme $\mathfrak{Y}$ 
if and only if there exists a partition $\{\Delta_i\}_{i=0}^e$ of $\{0,1,\ldots,d\}$
with $\Delta_0=\{0\}$ such that 
each $(\Delta_i,\Lambda_j)$ block of the first eigenmatrix $P$ has a constant row sum.
Moreover, the constant row sum of the $(\Delta_i,\Lambda_j)$ block 
is the $(i,j)$ entry of the first eigenmatrix of $\mathfrak{Y}$.

\medskip

Let $\mathfrak{X}$ be an association scheme given in Theorem~\ref{thm:1}.
Fusion schemes of $\mathfrak{X}$ with at least three classes 
are listed in Table~\ref{table}.
\begin{table}[htb]
 \begin{center}
  \begin{tabular}{|c|c|c|c|}
  \hline
  &fused relations & class & nonsymmetirc or symmetric \\
  \hline\hline
  $\mathfrak{X}_1$&$\{1,2\}$ & $5$ & nonsymmetric \\
  \hline
  $\mathfrak{X}_2$&$\{3,4\}$ & $5$ & nonsymmetric \\
  \hline
  $\mathfrak{X}_3$&$\{1,2\},\{3,4\}$ & $4$ & symmetric \\
  \hline
  $\mathfrak{X}_4$&$\{3,4,6\}$ & $4$ & nonsymmetric \\
  \hline
  $\mathfrak{X}_5$&$\{1,2\},\{3,4\},\{5,6\}$ & $3$ & symmetric \\
  \hline
  $\mathfrak{X}_6$&$\{1,2,3,4\}$ & $3$ & symmetric \\
  \hline
  $\mathfrak{X}_7$&$\{1,3\},\{2,4\},\{5,6\}$ & $3$ & nonsymmetric \\
  \hline
  $\mathfrak{X}_8$&$\{1,4\},\{2,3\},\{5,6\}$ & $3$ & nonsymmetric \\
  \hline
  \end{tabular}
  \caption{Fusion schemes of $\mathfrak{X}$}
  \label{table}
 \end{center}
\end{table}
The first eigenmatrix of $\mathfrak{X}_7$ and $\mathfrak{X}_8$ in Table~\ref{table}
are the matrix (\ref{d3P1}) by putting $a=b$. Put $b=4^p$.
We can verify that $\mathfrak{X}_7$ and $\mathfrak{X}_8$ for $p=3$ are isomorphic,
and $\mathfrak{X}_7$ and $\mathfrak{X}_8$ for $p=5$ are not isomorphic.
The computation needed to verify these facts was done with the
help of Magma \cite{magma}.

\section{Complex Hadamard matrices and Galois rings}\label{sec:5}

We continue to use the notation introduced in Section~\ref{sec.Galoisring}.
Let $\{A_j\}_{j=0}^6$ be the set of the adjacency matrices of $\mathfrak{X}$.
Then $A_1^T=A_2$, $A_3^T=A_4$, and $A_5,A_6$ symmetric.
We call the algebra $\mathfrak{A}=\langle A_0,A_1,\ldots,A_6\rangle$ 
the Bose--Mesner algebra of $\mathfrak{X}$.
Our next theorem gives a classification of 
hermitian complex Hadamard matrices belonging to $\mathfrak{A}$.

\begin{thm}\label{thm:2}
Let $w_0=1$ and $w_j$ {\rm(}$1\leq j\leq6${\rm)} be complex numbers of absolute value $1$.
Set
\[ W=\sum_{j=0}^6w_jA_j\in\mathfrak{A}, \]
and assume that $W$ is hermitian.
Then, $W$ is a complex Hadamard matrix if and only if 
\begin{align}
W=&A_0+\epsilon_1i(A_1-A_2)+\epsilon_2i(A_3-A_4)+A_5+A_6, \quad \text{or} \label{eq:W1}\\
W=&A_0+\epsilon_1i(A_1-A_2)+\epsilon_2(A_3+A_4)+A_5-A_6,\label{eq:W2}
\end{align}
for some $\epsilon_1,\epsilon_2\in\{\pm1\}$.
\end{thm}

\begin{proof}
Recall that $\{A_j\}_{j=0}^6$ is the set of the adjacency matrices of $\mathfrak{X}$ given in Theorem~\ref{thm:1}.
Notice that $A_1^T=A_2$, $A_3^T=A_4$, while $A_5$ and $A_6$ are symmetric.
Since $W$ is hermitian, we have $w_1w_2=1$ and $w_3w_4=1$.

Suppose that the matrix $W$ is a complex Hadamard matrix.
Then $(w_i)_{i=1}^6$ is a common zero of the polynomials $e_k$ ($k=1,\ldots,6$) defined in \eqref{eq:ek}.
Since
\begin{equation}\label{01-7}
e_2-e_1=4ib(X_1-X_2)(X_5-1),
\end{equation}
we have $w_1=w_2$ or $w_5=1$.

Suppose first that $w_1=w_2$. After specializing $X_1=X_2$ and $X_1^2=1$, 
we have
$$e_1=(X_5-(2b+1))(X_5+2b-1).$$
Then $|w_5|=|2b\pm1|\geq 2b-1>1$, which is a contradiction.
Therefore, we must have $w_5=1$.
After specializing $X_5=1$, we have
$$e_1=-b^2(X_1-X_2-2i)(X_1-X_2+2i).$$
Hence $w_1=-w_2\in\{\pm i\}$.
Moreover, after specializing $X_5=1$, $X_1=-X_2$ and $X_1^2=-1$, we have
\begin{align}
e_4-e_3&=4ib^2(X_6-1)(X_3-X_4),\label{01-8}\\
e_5-e_6&=4b^2(X_6+1)(X_3+X_4).\label{01-9}
\end{align}
If $w_6=1$, then $w_4=-w_3\in\{\pm i\}$ by \eqref{01-9}.
Therefore we have \eqref{eq:W1}.
If $w_3=w_4$, then $w_3\in\{\pm1\}$ and $w_6=-1$ by \eqref{01-9}.
Therefore we have \eqref{eq:W2}.

Conversely, assume that $W$ is one of the matrices \eqref{eq:W1}, \eqref{eq:W2}.
We show that $W$ is a complex Hadamard matrix.
By Lemma~\ref{lem:equiv}, it suffices to show that
$(w_i)_{i=1}^6$ is common zero of the polynomials \eqref{eq:ek}, and
it is easy to do this.
\end{proof}

We note that the matrix \eqref{eq:W1} 
belongs to the Bose--Mesner algebra of  $\mathfrak{X}_7$ or $\mathfrak{X}_8$  in Table~\ref{table},
depending on $\epsilon_1=\epsilon_2$ or $\epsilon_1=-\epsilon_2$.
Also, the matrix \eqref{eq:W2}
belongs to the Bose--Mesner algebra of $\mathfrak{X}_2$ or $\mathfrak{X}_4$ in Table~\ref{table},
depending on $\epsilon_2=1$ or $\epsilon_2=-1$.
Therefore, no proper fusion scheme of $\mathfrak{X}$ contains 
all the matrices \eqref{eq:W1} and \eqref{eq:W2} in its Bose--Mesner algebra $\mathfrak{A}$.
We also note that Ma \cite{Ma} considered association schemes which are
invariant under the multiplication by $\T$. 
The only association schemes
in Table~\ref{table} which are invariant under the multiplication by $\T$
is $\mathfrak{X}_7$, and it has the first eigenmatrix described by
\cite[Theorem 7]{Ma}.


\end{document}